\documentclass[12pt]{amsart}

\usepackage[margin=1.15in]{geometry}
\usepackage{amscd,amssymb, amsmath, setspace}
\usepackage{url}
\usepackage{hyperref}
\usepackage{xspace}
\usepackage{graphicx}
\usepackage{enumerate}
    \usepackage{verbatim}

\theoremstyle{plain}
\newtheorem{theorem}{Theorem}[section]
\newtheorem{prop}[theorem]{Proposition}

\newtheorem{cor}[theorem]{Corollary} 

\newtheorem{lemma}[theorem]{Lemma}

\theoremstyle{definition}

\newtheorem{defn}[theorem]{Definition}
\newtheorem{rmk}[theorem]{Remark} 
\newtheorem{rmrk}[theorem]{Remark} 
\newtheorem{eg}[theorem]{Example}
\newtheorem{ex}[theorem]{Example}
\newtheorem{problem}[theorem]{Problem}
\newtheorem*{ex*}{Example}
\newtheorem*{conj*}{Conjecture}
\newtheorem*{thmNoNum}{Theorem}
\newtheorem*{claim*}{Claim}

\DeclareMathOperator{\link}{link}
\DeclareMathOperator{\init}{init}

\DeclareMathOperator{\supp}{supp}

\renewcommand{\subset}{\subseteq}

\newcommand{\ZZ}{\ensuremath{\mathbb{Z}}}

\newcommand{\RR}{\ensuremath{\mathbb{R}}}

\newcommand{\mtx}{\begin{bmatrix}} 
\newcommand{\mtxend}{\end{bmatrix}}

\newcommand{\ideal}[1]{\left(#1\right)}
\newcommand{\isomorphism}{\cong}

\renewcommand{\vec}[1]{\mathbf{#1}}
\newcommand{\abs}[1]{\lvert#1\rvert}

\newcommand{\restrict}[2]{{#1}\vert_{#2}}

\newcommand{\defeq}{\mathrel{\mathop{:}\hspace{-0.75ex}=}}
\newcommand{\eqdef}{\mathrel{=\hspace{-0.75ex}\mathop{:}}}

\begin{document}

\title[Betti numbers and Markov degrees]{Betti numbers of Stanley-Reisner rings determine hierarchical Markov degrees}

\author{Sonja Petrovi\'c}
\address{Sonja Petrovi\'c is with the Department of Statistics, The Pennsylvania State University, University Park PA 16802
		\\\emph{Current address:} Statistics Department, University of Chicago.}
\email{\url{petrovic@psu.edu}}

\author{Erik Stokes}
\address{Erik Stokes lives in Odenton MD}
\email{\url{stokes.erik@gmail.com}}

\date{\today}

\begin{abstract} 
There are two seemingly unrelated ideals associated with a simplicial
complex $\Delta$: one is the Stanley-Reisner ideal $I_\Delta$, the
monomial ideal generated by minimal non-faces of $\Delta$, well-known
in combinatorial commutative algebra; the other is the toric ideal
$I_{M(\Delta)}$ of the facet subring of $\Delta$, whose generators
give a Markov basis for the hierarchical model defined by $\Delta$,
playing a prominent role in algebraic statistics.

In this note we show that the complexity of the generators of
$I_{M(\Delta)}$ is determined by the Betti numbers of $I_\Delta$.  The
unexpected connection between the syzygies of the Stanley-Reisner
ideal and degrees of minimal generators of the toric ideal provide a
framework for further exploration of the connection between the model
and its many relatives in algebra and combinatorics.
 \end{abstract}

\maketitle

\section{Introduction}
\label{section:intro}

A central problem in algebraic statistics is the study of the combinatorial properties and complexity of Markov bases for toric models.  
In statistics, Markov bases  provide an alternative, non-asymptotic approach to performing goodness-of-fit tests and model selection. Their use has increased in recent years,  especially for the models where the standard tools do not scale well.  
 A statistical model is called {algebraic} if its  parameter space is a semi-algebraic set.  In this case, the model corresponds to the real positive part of the algebraic variety obtained by taking the Zariski closure of the image of the model parametrization map. 
When the variety is toric,  any generating set of its defining  ideal is a {Markov basis} for the model;  this is a fundamental theorem  that appeared in the breakthrough paper \cite{DiacSturm}.
From a large and growing literature in algebraic statistics, we single out three recent books that can serve as an overview: \cite{algStatBook}, \cite{ASCBbook} and \cite{italianAlgStatBook}.

Simplicial complexes define an important class of toric models called hierarchical models (defined in \cite{fienberg} and see also
 \cite{HoSu}). In a hierarchical model, relationships  between $m$ discrete random variables are described by a simplicial complex  $\Delta$ on $m$ vertices: 
 the facets of $\Delta$ determine which margins of the corresponding $m$-way contingency table, or tensor, serve as minimal sufficient statistics for the model. The matrix $M(\Delta)$ of the sufficient statistics facet-margin linear map defines a toric ideal $I_{M(\Delta)}$ (see Section~\ref{sec:background} for a precise definition).

A general open problem is to  better understand how combinatorial (and other) properties of the simplicial complex determine  the properties of the toric ideal  $I_{M(\Delta)}$ of the hierarchical model. 
A crucial question is to determine the combinatorial complexity and degree estimates for the generators of the toric ideal. This problem has been studied for several families of models, but remains open in general. 
In algebraic statistics, the largest degree of a minimal generator of $I_{M(\Delta)}$ is called the \emph{Markov width} of the model, as it provides the bound on the complexity of the moves needed for the Markov Chain used to walk on the space of observations.  
   There are several interesting results in this direction, giving upper and lower bounds and solving this problem for special cases (for example \cite{slimTables} and  \cite{HoSu}, and  for related work, see \cite{aokiTakemura}, \cite{haraTakemuraYoshida}). 
    In addition, if $\Delta$ is a reducible complex, then $I_{M(\Delta)}$ is  a toric fiber product, so it is possible to lift the generators  inductively (for the most general construction that applies to reducible complexes, see  \cite{seth:ToricFiberProducts} and \cite{sethAlexThomas}). 
   
The Stanley-Reisner ideal is a well-studied ideal associated to a
simplicial complex $\Delta$ on $m$ vertices; it is the squarefree
monomial ideal $I_\Delta \subset K[x_1,\dots,x_m]$ generated by the
minimal non-faces of $\Delta$.  Betti numbers of Stanley-Reisner
ideals have been studied by many authors from different points of
view.  Our study uses some known
  results and combinatorial techniques to provide a first link between Betti
numbers of $I_\Delta$ and the degrees of minimal generators of
$I_{M(\Delta)}$.

Our main result shows that the syzygies of the Stanley-Reisner ideal predict degrees of minimal generators of the toric ideal of the hierarchical model:
\begin{thmNoNum} [Theorem~\ref{thm:main}]
	Let $\Delta$ be any simplicial complex.  Suppose that the
        Betti diagram of the minimal free resolution of $R/I_{\Delta}$
        has a non-zero entry in the $j$-th row (that is
        $\beta_{i,i+j}(R/I_\Delta)\neq 0$ for some $i$).  Then the
        toric ideal of the hierarchical model, $I_M$, has a minimal
        generator with degree $2^j$.
\end{thmNoNum}
The proof of Theorem~\ref{thm:main} is in
Section~\ref{section:theorems}, which also contains a geometric
interpretation of the result.  In
Section~\ref{section:moves-explicitly} we outline a way to construct
some of the predicted Markov moves from Theorem~\ref{thm:main}.  We
also carry out a computational summary of unpredictable moves in
Section~\ref{sec:unpredictable}.

One important family of complexes that arises in the theory of
graphical models is \emph{decomposable complexes}.  For a detailed
study, see for example \cite{Dobra2003} and \cite{MR1789526}, where
they appear under the name \emph{decomposable graphical models}.  They
correspond to the hierarchical models of clique complexes of chordal
graphs (cf. Theorem 3.3.3.\ in \cite{algStatBook}).  In
Corollary~\ref{thm:decomposable}, we recover a result of Fr\"oberg on
linear resolutions for decomposable complexes.

The Betti diagram of $I_\Delta$ can vary with the ground field for general $\Delta$. In particular, the regularity may change. Recall that the regularity is defined to be the largest $j$ so that $\beta_{i,i+j}(R/I_\Delta)\neq 0$, that is, the number of rows in the Betti diagram. For example, the Alexander dual of any triangulation of the
projective plane changes regularity when the field has characteristic $2$. 
Our results hold over \emph{any} field. 
 Therefore, we can choose a ``worst-case'' field, where the regularity is largest, thus giving us most information about the Markov complexity of the hierarchical model.

The insight provided by the Betti numbers of the Stanley-Reisner ideal motivates interpretation of other classical numerical invariants of the coordinate ring.  In terms of the model, these questions are still unexplored (Section~\ref{sec:unpredictable} contains a few open problems).  We hope this relationship will inspire a further study of the effect of the algebraic and geometric invariants of the classical combinatorial object on the underlying algebraic statistical model. 

Before providing technical details, we illustrate the main Theorem on an example. 
\begin{ex}\label{ex:4-cycle}
	Consider the binary model of a $4$-cycle $\Delta$ with facets
        $ \{12\}, \{23\}, \{34\}, \{14\} $.  The Stanley-Reisner ideal
        $(x_1x_3, x_2x_4)$ has the following Betti diagram:
	\begin{center}
	\begin{tabular}{rrrrc}
	 &       0 & 1 & 2&\\
	total: & 1 & 2 & 1&\\
	    0: & 1 & . & .&\\
	    1: & . & 2 & .&\\
	    2: & . & . & 1&.
	\end{tabular}
	\end{center}
Since rows $1$ and $3$ have nonzero entries,	Theorem~\ref{thm:main} states that the toric ideal $I_M$        must have generators in degrees $2^1=2$ and $2^2=4$.  

The toric ideal $I_M$ lives in the polynomial ring with $16$ variables $p_{i_1i_2i_3i_4}$, $i_j\in\{0,1\}$, which refer to the binary states of four random variables corresponding to the vertices of $\Delta$. 
In this case, $I_M$ is generated solely in         degrees $2$ and $4$:
	\begin{align*}
		I_M =  (&p_{1011} p_{1110} - p_{1010} p_{1111},\phantom{xxx}
                 p_{1001} p_{1100} - p_{1000} p_{1101},\\
		& p_{0111} p_{1101} - p_{0101} p_{1111}, \phantom{xxx}
		p_{0110} p_{1100} - p_{0100} p_{1110}, \phantom{xxx}
		p_{0011} p_{0110} -p_{0010} p_{0111},\\
                & p_{0011} p_{1001} - p_{0001}p_{1011},\phantom{xxx}
                 p_{0001} p_{0100} - p_{0000} p_{0101},\phantom{xxx}
                 p_{0010} p_{1000} - p_{0000} p_{1010},\\
                p_{0100}&p_{0111} p_{1001} p_{1010} - p_{0101} p_{0110} p_{1000} p_{1011},\phantom{xxx}
                p_{0010} p_{0101} p_{1011}p_{1100} - p_{0011} p_{0100} p_{1010} p_{1101},\\
                p_{0001} &p_{0110} p_{1010} p_{1101} - p_{0010}p_{0101} p_{1001} p_{1110},\phantom{xxx}
                p_{0001} p_{0111} p_{1010} p_{1100} -p_{0011} p_{0101} p_{1000} p_{1110},\\
                p_{0000} &p_{0011} p_{1101} p_{1110} - p_{0001} p_{0010}p_{1100} p_{1111},\phantom{xxx}
                p_{0000}p_{0111} p_{1001} p_{1110} - p_{0001} p_{0110} p_{1000} p_{1111},\\
                p_{0000} &p_{0110} p_{1011} p_{1101} - p_{0010} p_{0100} p_{1001}p_{1111},\phantom{xxx}
                 p_{0000} p_{0111}p_{1011} p_{1100} - p_{0011} p_{0100} p_{1000} p_{1111}).
        \end{align*}
\end{ex}

\section{Toric ideals of hierarchical models}
\label{sec:background}

\label{section:models}

Hierarchical models generalize the notion of row and column sums of a matrix.  For higher dimensional tensors, the models are defined in terms of the facets of some simplicial complex.
\begin{defn}\label{def:hier-model-map}
	\par\noindent
	\begin{enumerate}[(a)]
	\item Given a $d_1\times\dotsm\times d_n$ table
          $T\in\bigotimes_{j=1}^n\RR^{d_j}$ and $F\subset [n]$,
          define the \emph{$F$-margin}
          of $T$ to be
	\[T_F \defeq \sum_{(i_j\mid j\not\in F)} T_{i_1\dotsc
	  i_n}\in\bigotimes_{j\in F} \RR^{d_j}.\]

	\item Let $\Delta$ be a simplicial complex on $[n]$, with
	  facets $F_1,\dotsc,F_s$ and $d=(d_1,\dotsc,d_n)$. Define a
	  linear map $\phi_\Delta=\phi_{\Delta,d}\colon           \bigotimes_{i=1}^n
	  \RR^{d_i}\to \bigoplus_{i=1}^s(\bigotimes_{j\in             F_i}\RR^{d_j})$
	  sending a table $T$ to $(T_{F_1},\dotsc,T_{F_s})$.  
	  $M(\Delta,d)$ is the matrix representing this map in the           standard basis.
	\end{enumerate}
\end{defn}

\begin{ex}\label{ex:segre}
	The matrix for the complex with facets $\{1\}$ and $\{2\}$ is 
	\[\left[\begin{array}{rrrr}
	1 & 1 & 0 & 0 \\
	0 & 0 & 1 & 1 \\
	1 & 0 & 1 & 0 \\
	0 & 1 & 0 & 1
	\end{array}\right].\]
	The map $\phi:\RR^4\to\RR^2\oplus \RR^2$, restricted to the
        probability simplex, sends a $2\times 2$ table $[T_{ij}]$ to
        the 1-dimensional margins obtained by summing over each index:
        $[T_{ij}]\mapsto ([\sum_j T_{ij}]_{i=1,2},[\sum_i
          T_{ij}]_{j=1,2})$.  These are the row and column sums of the
        table.
\end{ex}

Hierarchical models are a subclass of log-linear models in statistics: sets  of all probability distributions whose logarithms are in the linear span of the rows of some matrix $A$.
The hierarchical model $\mathcal M_{M(\Delta)}$ is the intersection of the toric variety parametrized by the map induced by $M(\Delta)$ and the probability simplex. 
In the language of algebraic statistics, the facets of the complex determine the sufficient statistics of the model. 

In practice, one is interested in performing a random walk on the fibers  of the model, defined as the sets of points with the same sufficient statistics.  Such random walks are crucial for testing the goodness of fit for the model. 
 The \emph{fiber} of a point $u$ in the image of $M(\Delta)$, denoted by $\mathcal F(u)= \phi_\Delta^{-1}(u)$, is a set of all points $v$ such that $M(\Delta) v =  M(\Delta) u$.
A \emph{Markov basis} $\mathcal B\subset \ker M(\Delta)$ for the model $\mathcal M_{M(\Delta)}$ is a finite set of tensors, called  \emph{moves}, 
that connects every fiber $\mathcal F(u)$ in the following sense: if $T_1,T_2\in\mathcal{F}(u)$ then $T_1=T_2+\sum_{i=1}^k m_i$ and $T_2+\sum_{i=0}^\ell m_i \in\mathcal{F}(u)$ for every $\ell\leq k$, for some collection of moves $\{m_1,\dots,m_k\}\subset\mathcal B$.

The matrix $M(\Delta)$ of a hierarchical model  determines the toric ideal 
$
	I_{M(\Delta)}\defeq(x^u-x^v : u-v\in\ker M(\Delta))
$
in the polynomial ring $S\defeq K[x_1,\dots,x_N]$, where $N$ is the
number of entries in the table.  A starting point of the field of algebraic statistics is the realization that generators of the ideal  form a Markov basis for the model (\cite{DiacSturm}, see also Theorem 1.3.6.\ in \cite{algStatBook}). 
\begin{ex}
	Extending the map for the $2\times 2$ table from Example
        \ref{ex:segre} to complex numbers as follows:
	\begin{align*}
		\tilde\phi: \mathbb{C}[T_{11},T_{12},T_{21},T_{22}] &\to \mathbb{C}[r_1,r_2,c_1,c_2]\\
			T_{ij} &\mapsto r_i c_j,
	\end{align*}
	we see that the image of $\tilde\phi$ restricted to the
        probability simplex is exactly the image of $\phi$; 
	namely, $T_{ij}$ represents the $(i,j)$-entry in the table. 
	The toric ideal
        $\ker\tilde\phi$ is generated by $T_{11}T_{22}-T_{12}T_{21}$,
        the determinant of the generic $2\times 2$ matrix. Therefore, this hierarchical model is the real positive part of the Segre embedding  $\mathbb        P^1\times \mathbb P^1 \to \mathbb P^3$.
\end{ex}

In algebraic statistics, we refer to the degrees of minimal generators of the toric ideal $I_{M(\Delta)}$ as \emph{Markov degrees}. 
Surprisingly, we make a link between these Markov degrees and Betti numbers of the monomial ideal $I_\Delta$. 

\section{Syzygies of Stanley-Reisner ideals}
\label{section:SR}
Here we briefly recall the necessary background and notation. 
Let $K$ be any field, $R\defeq K[x_1,\dotsc,x_n]$ and define $x_\sigma \defeq\prod_{i\in\sigma}x_i$ for $\sigma\subset [n]=\{1,\dotsc,n\}$. Then the \emph{Stanley-Reisner ideal} of the simplicial complex $\Delta$ is
\[
	I_\Delta\defeq \ideal{x_\sigma\mid \sigma\not\in\Delta}\subset R.
\]

There are three basic constructions for a simplicial complex: restrictions, deletions, and links.  
For $\sigma\subset [n]$, the \emph{restriction} of $\Delta$ to $\sigma$ is defined as
$
	\restrict{\Delta}{\sigma}:=\{F\in\Delta\mid F\subset \sigma\}.
$
The \emph{deletion} of $v$ from $\Delta$ is
$	
	\Delta_{-v}\defeq \{F\in\Delta\mid v\not\in F\}. 
$
The \emph{link} of a complex with respect to $F\subset [n]$ is defined to be
$
	\link_\Delta(F)\defeq\{G\in\Delta\mid G\cap F=\emptyset,\; F\cup G\in\Delta\}.
$

 The \emph{graded Betti numbers} $\beta_{ij}$ of $I_\Delta\subset R$ encode the ranks of the syzygy modules in a minimal free resolution of its coordinate ring $R/I_\Delta$:
 \[
	0\to \bigoplus_{j\in\ZZ} R(-j)^{\beta_{pj}}\to \dotsm \to \bigoplus_{j\in\ZZ} R(-j)^{\beta_{2j}}\to \bigoplus_{j\in\ZZ} R(-j)^{\beta_{1j}}\to R\to R/I\to 0.
\]
Recall that $\beta_i
(R/I_\Delta)\defeq \sum_j \beta_{ij}(R/I_\Delta)$ are the \emph{total
  Betti numbers}.  If $I_\Delta$ and $R$ are understood from the
context, we simply write $\beta_{ij}$ for the graded Betti numbers.
We can also grade $R$  by
$\ZZ^n$ instead of $\ZZ$ by setting $\deg x_i$ to the $i$-th standard
unit vector in $\ZZ^n$.  Then the summands of the $i^{th}$ syzygy module  have the form $R(-\vec{b})^{\beta_{i\vec{b}}}$
for some $\vec{b}\in\ZZ^n$.  We call $\vec{b}$ a \emph{multidegree}
and $\beta_{i\vec b}$ a \emph{multigraded} Betti number.
Typically, we summarize this numerical
data in a standard \texttt{Macaulay2} \cite{M2} Betti diagram, a table
whose $(i,j)$-th entry is $\beta_{i,i+j}$.
	\begin{center}
	\begin{tabular}{rrrr}
	 &       $\dotsm$ & $i$ \\
	total: & $\dotsm$ & $\beta_i$ \\
            $\vdots$ &&$\vdots$ \\
	    $j$: & $\dotsm$ & $\beta_{i,i+j}$ \\
	\end{tabular}
	\end{center}
This is the notation used in Example~\ref{ex:4-cycle}, where, in the
interest of readability, we use $\cdot$ in place of 0. 

The fundamental result describing the Betti numbers of Stanley-Reisner ideals is Hochster's formula, which relates the Betti numbers to the simplicial cohomology of the complex.  
\begin{theorem}[Hochster's Formula, \cite{miller05:_combin_commut_algeb}, Corollary~5.12]
	The graded Betti numbers of $R/I_\Delta$ are given by
	\[
		\beta_{i,i+j}(R/I_\Delta)= \sum_{\abs{\sigma}=j}	\dim_K\widetilde{H}^{j-1}(\restrict{\Delta}{\sigma};K).
	\]
\end{theorem}

\section{From Betti diagrams to Markov degrees}
\label{section:theorems}

\subsection{Initial degrees}
 
In the special case of the initial Markov degrees, we can prove a stronger statement which is not true in the general case. 
The \emph{initial degree} of a homogeneous ideal $I$, denoted by $\init(I)$, is the smallest degree of a minimal generator of $I$.   
To state the relationship between $\init(I_\Delta)$ and $\init(I_{M(\Delta)})$, we need the following constructions. 

In \cite{HoSu}, Ho\c{s}ten and Sullivant observe that  $M(\Delta,d)$ is of the following form:
\begin{align*}
	M(\Delta,d)=\mtx A &0 &0 &\dots &0 \\ 0 &A &0 &\dots &0 \\ \vdots
	&\vdots &\ddots &\vdots &\vdots \\ 0 &0 &0 &\dots &A\\ B &B &B &\dots &B \mtxend ,
\end{align*}
where there are $d_1$ copies of $A$ and $B$.
They also note that $A$ and $B$ are matrices corresponding to
$\link_\Delta(1)$, and to the complex generated by the facets that do
not contain the vertex $1$, respectively.  Interestingly, even though
any $M$ can be built from $A$ and $B$, not every complex can.
Instead, an arbitrary complex can be built from links and deletions,
so there is a subtle but crucial difference between $\Delta_{-v}$ and
 the complex defined by
$B$. 

\paragraph{\emph{{\bf Notation.  }}}
From here on, we adopt the following notation: $\Delta$ is a simplicial complex on $n$ vertices, $d=(d_1,\dotsc,d_n)$ is a sequence of non-negative integers and $M=M(\Delta,d)$.  In addition, $A$ and $B$ are the matrices from the above decomposition of $M$. The toric ideal of the model is denoted by $I_M$ and lives in the polynomial ring $S$.

It is interesting to note that in some special cases, the matrix of the hierarchical model has a structure that makes an appearance in integer programming.  Namely, if $\dim\Delta=d$ and the only facet not containing $1$ has dimension $d$, then $B$ is the identity matrix, and $M$ becomes a \emph{higher Lawrence lifting} of $A$.  Then the Graver basis of $I_M$ coincides with any minimal Markov basis and any reduced Gr\"obner basis~\cite[Theorem~7.1]{St}.
 
\smallskip
Initial degrees of
generators of $I_M$ and $I_\Delta$ are related as follows.
\begin{prop}
\label{prop:initial-degree}
\label{lemma:sf-init-degree}
   Let $k=\init(I_\Delta)$.  Then $\init(I_M)=2^{k-1}$ provided    that $d_i\geq 2$ for $1\leq i\leq n$.

   Furthermore, let $D=\init (I_M)$.  Then, if    $(u_1,\dotsc,u_{d_1})\in\ker M$ with $u_i\in\ker A$ has degree $D$,    it is squarefree.
\end{prop}

\begin{proof}
We proceed by induction on $k$.  If $k=1$, then there is some vertex
not in $\Delta$, and $I_M$ contains a linear form. Therefore
$\init(I_M)=1$.  To see this, without loss of generality, suppose that
$\{1\}\not\in\Delta$.  Then, we may write $M=[B\ \dotsm B]$. Clearly,
there is a linear form in $\ker M$. For example,
$(1,0,\dotsc,0,-1,0,\dotsc,0)\in\ker M$, where the $-1$ is in position
$\prod_{i=2}^n d_i+1$.

Suppose $k>1$.  Let $v$ be a vertex in the support of a minimal
generator of $I_\Delta$ with degree $k$.  Then, the Stanley-Reisner
ideal of $\link_\Delta(v)$, regarded as a complex over $[n]-\{v\}$, is
$((I_\Delta\colon x_v)+x_v)/\ideal{x_v}\subset R/\ideal{x_v}$.  Due to
our choice of $v$, this ideal has initial degree $k-1$.  For
convenience, we relabel the vertices of $\Delta$ so that $v$ is
labeled as vertex $1$, noting that this does not change the initial
degrees of $I_\Delta$ or $I_M$.

By induction, $I_A$ has initial degree $2^{k-2}$.  Let $u\in\ker A$
have minimal degree.  Then $(u,-u,0,\dotsc,0)\in\ker M$ has degree
$2\cdot 2^{k-2}=2^{k-1}$.  We claim that this is the smallest degree
of any element of $\ker M$.  Suppose, to the contrary, that there is
some $(a_1,a_2,\dotsc,a_{d_1})\in\ker M$ with degree less than
$2^{k-1}$.  Then at least one of the $a_i$'s must have degree less
than $2^{k-1}/d_1\leq 2^{k-2}$.  Now $a_i\in\ker A$ provides a
contradiction.

To prove the second claim, we may assume that $\init(I_A)=D/2$.  Since the degree of any of the $u_i$'s can not be smaller than $D/2$, at most $2$ of the $u_i$'s are
non-zero: say, $u_{j}\eqdef u$ and $u_l \eqdef v$.  There are two cases: $\deg u=\deg v=D/2$, or one of the vectors is $0$ while the other has degree $D$.  In the first case, each of $u$ and $v$ must be in the Markov basis for $A$.  By induction on $n$, both are squarefree and thus the concatenation is also squarefree.

In the second case, since the nonlinear generators of the
Stanley-Reisner ideal of the deletion are also minimal generators of
the Stanley-Reisner ideal of $\Delta$, we see that having a degree $D$
element in $I_B$ forces $\init (I_B)=D$.  The claim now follows by
induction on the number of vertices.
\end{proof}

\begin{rmk}
   Results of Proposition~\ref{prop:initial-degree} have been
   proved in \cite[Theorem 5]{kahle-neighbor} using only linear
   algebra, independently and at the same time this manuscript was
   written.
\end{rmk}

The only case in which $I_M$ has initial degree 1 is when there is
some vertex not contained in $\Delta$.  In this case, we may simply
take quotients to pass to smaller polynomial rings, and thus assume
that $\init(I_M)\geq 2$.  It then follows from
Proposition~\ref{lemma:sf-init-degree} that all degree $2$ elements in
any Markov basis are squarefree.  The squarefree quadratic Markov
moves are studied in \cite{hara-mamimbodmfct} under the name
primitive moves.

\subsection{The general case}

The proof of the main theorem is by induction.  We use the following technical result  to reduce to smaller complexes.

\begin{lemma}\label{lemma:link-homology}
If $\Delta$ is a complex with  $H^d(\Delta;K)\neq 0$ and
$H^d(\Delta_{-v};K)=0$ for some $v\in\Delta$, then
$H^{d-1}(\link_\Delta(v);K)\neq 0$.
\end{lemma}
\begin{proof}
Define $\operatorname{star}_\Delta(v):=\{F\in\Delta\mid \{v\}\cup
  F\in\Delta\}$.  Note that
  $\operatorname{star}_\Delta(v)\cup\Delta_{-v}=\Delta$ and
  $\operatorname{star}_\Delta(v)\cap\Delta_{-v}=\link_\Delta(v)$.
  Applying the Mayer-Vietoris sequence to the above pair we get an
  exact sequence
\[\rightarrow H^{i-1}(\link_\Delta(v))\rightarrow
H^{i}(\Delta)\rightarrow H^{i}(\operatorname{star}_\Delta(v))\oplus
H^{i}(\Delta_{-v})\rightarrow . \]
Since, by assumption, $H^d(\Delta_{-v})=0$, and since the star is, by
definition, a cone with apex $v$, and thus contactable, the right-most term is $0$.  Since the middle
term is non-zero, it follows that the left-most term is also non-zero,
as required.
\end{proof}

The proof of the main theorem will make use of \emph{tableaux}
notation, which can be found, for example, in \cite[Chapter
  1]{algStatBook}.  We consider the elements of the Markov basis as
binomials in $K\bigl[p_b\mid b\in [d_1]\times\dotsm[d_n] \bigr]$ and
write the indices as the rows of two matrices.  For example, $p_{0111}
p_{0001}-p_{0011} p_{0101}$ is written
\[\begin{bmatrix}
	0 & 1 & 1 & 1\\
	0 & 0 & 0 & 1
\end{bmatrix}-\begin{bmatrix}
	0& 0& 1& 1\\
	0& 1& 0& 1
\end{bmatrix}.\]

Note that, by definition, a tableaux is in $I_M$ if and only if for every face, $F\in\Delta$, restricting the matrices to only those columns whose indices appear in $F$ gives 0 (that is, the two matrices are equal up to permutations of the rows).  In particular, if $i\in\Delta$ and column $i$ of the left-hand matrix is constant, then it must equal the corresponding column of the right-hand matrix.  We will make use of this fact in the proof below: we will ignore columns that are constant and pass to $\Delta_{-v}$. 

\begin{theorem}\label{thm:main}\label{thrm:main-thrm-new}

Let $\Delta$ be a complex with $\beta_{i,i+j}(R/I_\Delta)\neq 0$ for
some $i$.  Then
$M(\Delta)$ has a minimal Markov move with degree $2^j$.
\end{theorem}
\begin{proof}
Choose $i$ to be minimal so that $\beta_{i,i+j}(R/I_\Delta)\neq 0$.
Consider the multi-graded minimal free resolution of $R/I_\Delta$.
There is some $\vec{b}\in\{0,1\}^n$ with $\abs{\vec{b}}=i+j$ and
$\beta_{i,\vec{b}}(R/I_\Delta)\neq 0$.  Letting 
$\sigma:=\supp(\vec{b})=\{k_1,\dotsc,k_{i+j}\}$,  $H^{j-1}(\restrict{\Delta}{\sigma};K)\neq 0$ by Hochster's
formula. We consider 2 cases
depending on the size of $i+j$.

If $i+j<n$, then $\restrict{\Delta}{\sigma}$ is  a proper subcomplex of
$\Delta$.  By induction on the number of vertices, since the claim is trivial for 1 vertex, $M(\restrict{\Delta}{\sigma})$ has a minimal
Markov move $m$ of degree $2^j$.  Write $m$ in tableaux notation as
\[m=[m^+_{k_1}\cdots m^+_{k_{i+j}}]-[m^-_{k_1}\cdots m^-_{k_{i+j}}], \]
where each $m^{\pm}_k$ is a column vector and the rows are the indices
of the variables appearing in the positive and negative parts of the
binomial $m$, respectively.  Lift $m$ to 
\[\bar{m}=[\bar{m}^+_1\cdots \bar{m}^+_{n}]-[\bar{m}^-_1\cdots
  \bar{m}^-_{n}], \]
where $\bar{m}^+_k=\vec{0}$ if $k\not\in\sigma$ and
$\bar{m}^+_{k_s}=m^+_{k_s}$ otherwise. Clearly $\bar{m}\in
I_{M(\Delta)}$ and we need only show that it is minimal.  Suppose
to the contrary that $\bar{m}$ can be written as 
\begin{equation}\label{eq:2}
\bar{m}=\sum a_k v_k
\end{equation}
with $a_k\in S$ and $v_k\in I_{M(\Delta)}$ binomials.  Define the
\emph{support} of a binomial $v\in I_{M(\Delta)}$ to be
$\supp_\Delta(v)=\{k\mid v^+_k, v^-_k\text{ are not constant}\}$, where
$v^+_k$ is the $k$-th column of the tableaux form of $v$, as above.
Recall that $\sigma$ is the support of the shift $\vec{b}$ in the minimal free resolution.
If any of the $v_k$ in equation~\eqref{eq:2} have
$\supp_\Delta(v_k)\neq\sigma$, then each monomial in $v_k$ must be
canceled by some other summand with support $\supp_\Delta(v_k)$.  Thus
we may, without loss of generality, assume that
$\supp_\Delta(v_k)=\sigma$ for each $v_k$ in \eqref{eq:2}.  If
$\supp_\Delta(v_k)=\sigma$ then $v_k\in
I_{M(\restrict{\Delta}{\sigma})}$.  But then, after deleting the
columns not in $\sigma$, we have written $m$  as a
linear combination of elements of $I_{M(\restrict{\Delta}{\sigma})}$,
contradicting the minimality of $m$. Thus, $\bar{m}$ must be minimal
in $I_{M(\Delta)}$.  Since $\deg\bar{m}=\deg m=2^j$, we are done with
this case.

Suppose that $i+j=n$ (and thus $\vec{b}=(1,1,\dotsc,1)$).  By Hochster's
formula and the minimality of $i$, $H^{j-1}(\Delta;K)\neq 0$ and
$H^{j-1}(\Delta_{-v};K)= 0$ for every $v\in\Delta$.  By
Lemma~\ref{lemma:link-homology} $H^{j-2}(\link_\Delta(v);K)\neq 0$ and
thus (using Hochster's formula again)
$\beta_{i,i+j-1}R/(I_{\link_\Delta(v)}+x_v)\neq 0$.  By induction on
the number of vertices, $M(\link_\Delta(v))$ has a minimal Markov move, $m$,
of degree $2^{j-1}$.  Then, as in the proof
of Proposition~\ref{prop:initial-degree}, we can lift $m$ to a minimal Markov move 
$(m,-m,0,\dotsc,0)$ with degree $2^j$.
\end{proof}

\begin{rmrk}

Note that proof of Theorem~\ref{thrm:main-thrm-new} is, essentially, constructive.  Given
a non-zero multi-graded Betti-number, $\beta_{i,\vec{b}}$, restrict to
$\supp(\vec{b})$ and then link repeatedly until you arrive at a
complex with dimension 0, at which point the quadratic minimal Markov
moves can be explicitly described.  Then, lifting the move back to
$M(\Delta)$ as in the proof of Theorem~\ref{thrm:main-thrm-new} gives
a minimal Markov move for $\Delta$.  We will demonstrate this
procedure more explicitly in Section~\ref{section:moves-explicitly}.

\end{rmrk}

\medskip

Geometrically, the Betti diagram is determined by the reduced simplicial cohomology of the complex and its sub-complexes via Hochster's formula (see Section~\ref{section:SR}).

\begin{cor}\label{cor:markov-from-cohom}
	If  $\widetilde{H}^j(\Delta;K)\neq 0$ for some $j>0$ and some field $K$, 
	then $I_M$ has a degree $2^{j+1}$ minimal generator.
\end{cor}
\begin{proof}
	Taking quotients if necessary, we may assume that $f_0(\Delta)=n$.
	Hoschter's formula says that if we have a vector $a\in\{0,1\}^n$ and $A=\supp(a)$, then we can compute the $\ZZ^n$-graded Betti numbers by
	\[
		\dim\widetilde{H}^{\abs{A}-j-1}(\restrict{\Delta}{A};K)=\beta_{j,a}^K (R/I_\Delta) .
	\]
	Let $A=[n]$, the only set such that        $\restrict{\Delta}{A}=\Delta$.  Since all the Betti numbers of        $I_\Delta$ are squarefree  and $(1,\dotsc,1)$ is the only         squarefree integer vector with sum $n$, we get that $0\neq        \dim \widetilde{H}^j(\Delta;K)=\beta_{n-j-1,n}(R/I_\Delta)$.        Therefore, the Betti diagram has a non-zero entry in row        $n-(n-j-1)=j+1$.
\end{proof}

\begin{eg}\label{rmrk:homology-to-markov}
	Corollary~\ref{cor:markov-from-cohom} can be applied to all connected graphs.
	For example, if 	$\Delta$ is a cycle, or any graph containing a cycle as a vertex-induced
	subgraph, then $I_M$ must have a degree $4$ element in its Markov basis, because $\Delta$ has non-zero first cohomology.

	Similarly, if $\Delta$ is a simplicial
        $k$-sphere, then $H^k(\Delta;K)\isomorphism  K$. Hence the Markov basis
        must contain a degree $2^{k+1}$ element.  For example, any
        Markov basis of a hierarchical model associated to an
        octahedron contains at least one move with degree $8$.
\end{eg}

\begin{rmrk}\label{rmrk:k_4-minor}
  For graphs, it has recently been shown in \cite{arXiv:0810.1979v1} that the
  hierarchical model  is generated in degrees $2$ and $4$ if and only if
  the graph has no $K_4$-minors.  If $\dim\Delta=1$ then, using
  Hochster's formula, $\beta_{i,i+j}R/I_\Delta=0$ whenever $j>2$, so
  our theorem can only predict generators in degrees $2$ and $4$.  If
  $\Delta$ has a $K_4$-minor, then the dimension of
  $\widetilde{H}^1(\Delta;K)$ is at least that of
  $\widetilde{H}^1(K_4;K)$, which is $3$.
  Assuming that $\Delta$ is connected, this means our theorem can give
  an exact listing of generators only if the rightmost entry in the
  Betti diagram is at most $2$.  This fails in dimension $2$.
\end{rmrk}

\subsection{Linear resolutions}

 Theorem~\ref{thm:main} can be used to obtain information about
the Stanley-Reisner ideal from the Markov basis.  For example, if the
Markov basis contains elements of only $3$ distinct degrees, then the
Betti diagram of $I_\Delta$ can contain at most $3$ non-zero rows.
 In
fact, it is likely that there will be fewer then $3$, since
Theorem~\ref{thm:main} only includes the Markov basis
elements whose degree is a power of $2$.  As a special case,
the result 
provides a class of complexes whose Stanley-Reisner ideals have linear
resolutions.

\begin{cor}\label{cor:markov-linear-res} 
	If 
	  $I_M$ is generated in a single degree, then $I_\Delta$ has a linear
	resolution over {every} field.
\end{cor}

Decomposable complexes are clique complexes of chordal graphs. Theorem~4.4 in \cite{geiger-meek-sturmfels} characterizes decomposable models as those whose toric ideals are generated in degree $2$. This is one of the few classes of models where the generators are known. With that in mind, 
we obtain the following:
\begin{cor}\label{thm:decomposable}
	If $\Delta$ is decomposable then $I_\Delta$ has a 2-linear resolution over every field.
\end{cor}
This  recovers a theorem of Fr\"oberg. 
Namely,  clique complexes of chordal graphs are examples of flag complexes. In \cite{froberg-on-Stanley-Reisner-rings}, Fr\"oberg characterized the Stanley-Reisner rings of flag complexes that have a linear resolution: they are exactly clique complexes of chordal graphs.  
His proof uses  Mayer-Vietoris sequences and does not give a combinatorial formula for the Betti numbers. Recently, an easier proof of this proposition, along with a formula for the Betti numbers, appeared in \cite[Theorem 3.2.]{doch-eng-alg-properties}.

\section{Constructing the moves} 
\label{section:moves-explicitly}

In the previous section, we showed that the Betti diagram of the
Stanley-Reisner ideal implies the \emph{existence} of toric ideal generators of 
certain degrees.  Here we give a way to construct those binomials 
explicitly.  In general, this is a nontrivial task for any reasonably
complicated model.    We focus on binary hierarchical models, that
is, for the case when all random variables have two states.  The
general result follows similarly.

For a general recipe, we introduce notation to describe the columns of
the tableaux that we can construct using the procedure of the Theorem.
First, we define the alternating vectors $\varepsilon_k^\ell$ by
\[\begin{split}
\varepsilon_k^1&=[\underbrace{0 \dotsm 0}_{\text{$k$
    times}}\ \overbrace{1\dotsm 1}^{\text{$k$ times}} ]^T, \\
\varepsilon_k^\ell&=[\underbrace{\varepsilon_k^1\dotsm \varepsilon_k^1}_{\text{$\ell$
    times}}]^T\\
&=[\varepsilon_k^{\ell-1}\ \varepsilon_k^{1}]^T,
\end{split}
\]
so that $\varepsilon_2^1=[0\ 0\ 1\ 1]^T$ while
$\varepsilon_1^2=[0\ 1\ 0\ 1]^T$.  In general, $\varepsilon_k^\ell$
has alternating blocks of 0's and 1's; $k$ is the length of each block 
and $\ell$ is the number of blocks.

Additionally, we define column vectors $\alpha_i$ by $\alpha_1=[0]$
and $\alpha_i=[\alpha_{i-1}\ \overline{\alpha_{i-1}}]^T$, where
$\overline{\alpha}$ is the \emph{binary complement} of $\alpha$,
defined by $\overline{\alpha}_i := 1-\alpha_i$.  Here,
$\alpha_1=[0\ 1]^T$, $\alpha_2=[0\ 1\ 1\ 0]^T$,
$\alpha_3=[0\ 1\ 1\ 0\ 1\ 0\ 0\ 1]^T$, and so on.

\begin{prop}\label{prop:markov-recipe}
	Suppose that $j$ is minimal such that
        $\beta_{i,i+j}(R/I_\Delta)\neq 0$ and $\vec{b}\in\{0,1\}^n$ such
        that $\beta_{i,\vec{b}}$ with $\abs{\vec{b}}=i+j=d$.  If
        $\supp(\vec{b})=\{i_1,\dotsc,i_d\}$ with $i_1<i_2<\dotsm<i_d$,
        then $[u]-[v]$ is a minimal generator of $I_M$ with degree
        $2^{d-1}$ if, for $j<d$, the $i_j$-th columns of $u$ and $v$
        are $\varepsilon_{2^{d-j}}^{2^{j-1}}$, and the $i_d$-th
        columns of $u$ and $v$ are $\alpha_d$ and
        $\overline{\alpha_d}$, respectively.  The other columns of $u$
        and $v$ are equal to each other and each is either $[0\dotsm
          0]^T$ or $[1\dotsm 1]^T$.  There are a total of $2^{n-d}$
        such moves for each degree $d$ minimal generator of
        $I_\Delta$.
\end{prop}

\begin{proof}
  We induct on $n$.  If $n=1$ then there is only one complex with a
  non-trivial Stanley-Reisner ideal, $\Delta=\{\emptyset\}$.  One
  easily sees that the single binomial $[0]-[1]$ forms a minimal
  Markov basis.

Suppose that $n>1$. By permuting the vertices of $\Delta$, we may,
without loss of generality, assume that
$\vec{b}=(1,1,\dotsc,1,0,\dotsc,0)$.  If $i+j<n$ then, as in the proof
of Theorem~\ref{thrm:main-thrm-new}, consider the restriction
$\restrict{\Delta}{\supp(\vec{b})}$.  This has fewer vertices and so a
minimal Markov move $m$ of the claimed form exists by induction.
Then, using the notation of the proof of
Theorem~\ref{thrm:main-thrm-new}, $m$ can be lifted to a minimal
Markov move $\bar{m}$ of $I_M$.

        Now suppose that $i+j=n$ so that $\vec{b}=(1,1,\dotsc,1)$ and
        consider $\link_\Delta(1)$.  Let $[u']-[v']$ be the claimed
        minimal Markov move for $M(\link_\Delta(1))$ (which, as before, 
        exists by induction).  Then
	\[ \left[\begin{matrix}0 &
	    u'\\1&v'\end{matrix}\right]-\left[\begin{matrix}0 &
            v'\\1&u'\end{matrix}\right]\] is a minimal generator of
        $I_M$.   By construction, the first column of this
        matrix is $\varepsilon_{2^{d-1}}^1$, the $d$-th column is
        $\alpha_d$, while column $j$, for $2\leq j\leq d-1$, is
        $\varepsilon_{2^{d-j}}^{2^j}=[\varepsilon_{2^{d-j}}^{2^{j-1}}\ 
          \varepsilon_{2^{d-j}}^{2^{j-1}}]^T$.
	\end{proof}
\begin{ex}\label{ex:2}
	We can construct some quadratic Markov moves for the complex with facets
	\[
		\{\{12\},\{13\},\{14\},\{15\},\{23\},\{24\},\{25\},\{34\}\}.
	\]
	  This has Stanley-Reisner ideal $\ideal{x_{4} x_{5}, x_{3}
            x_{5}, x_{1} x_{2} x_{5}, x_{1} x_{3} x_{4}, x_{2} x_{3}
            x_{4}, x_{1} x_{2} x_{4}, x_{1} x_{2} x_{3}}$ and Betti
          diagram
	\[\begin{tabular}{rrrrr}
	 &        0 &  1 &  2 &  3\\
	total: &  1 &  7 & 10 &  4\\
	    0: &  1 &  . &  . &  .\\
	    1: &  . &  2 &  1 &  .\\
	    2: &  . &  5 &  9 &  4  
	\end{tabular} \qquad.\]
        The minimal entries on row 1 are the two quadratic minimal
        generators of $I_\Delta$, $x_3 x_5$ and $x_4 x_5$.  First,
        consider $x_3 x_5$ so that $\vec{b}=(0,0,1,0,1)$. 
        Then Proposition~\ref{prop:markov-recipe} shows that
	\[
		\left[\begin{matrix}0&0&0&0&0\\0&0&1&0&1\end{matrix}\right]-
		\left[\begin{matrix}0&0&0&0&1\\0&0&1&0&0 \end{matrix}\right]
	\] 
	is a minimal generator of $I_M$.  By altering this
        construction slightly, any of the first, second or fourth
        columns could be 1 instead of 0 (but must be the same on both
        sides), giving us a total of 8 minimal generators.  Starting
        with $x_4 x_5$ and following the same procedure gives us the
        minimal Markov move
	\[
		\left[\begin{matrix}0&0&0&0&0\\0&0&0&1&1\end{matrix}\right]-
		\left[\begin{matrix}0&0&0&0&1\\0&0&0&1&0 \end{matrix}\right].
	\]
	As before, the first, second and third columns can contain
        $1$'s instead of $0$'s, giving us $8$ more quadratic
        generators, for a total of $16$.  On the other hand, computing
        the Markov basis with \texttt{4ti2} gives that there are $24$
        quadratic moves in total. We will address this discrepancy
        shortly.

	Now, we construct a degree $4$
        minimal Markov move from the generator $x_1 x_3 x_4$.  The
        first column is $\varepsilon_2^1=[0\ 0\ 1\ 1]^T$, the third
        $\varepsilon_2^2=[0\ 1\ 0\ 1]$.  The the fourth columns of $u$
        and $v$ are $\alpha_2=[0\ 1\ 1\ 0]^T$ and
        $\overline{\alpha_2}=[1\ 0\ 0\ 1]^T$.  The other columns we
        may fill with 0's or 1's as we like.  We get
	\[\left[\begin{matrix}0&0&0&0&0\\0&0&1&0&1\\1&0&0&0&1\\1&0&1&0&0\\\end{matrix}\right]-
	\left[\begin{matrix}0&0&0&0&1\\0&0&1&0&0\\1&0&0&0&0\\1&0&1&0&1\\\end{matrix}\right], 
	\]
	along with $3$ other binomial obtained by varying columns 2
        and 4.  Iterating over all the cubic
	generators of $I_\Delta$ we construct a total of $20$ degree $4$
	binomials, out of  the $520$ that a minimal Markov basis contains. 
\end{ex}

\begin{rmrk}\label{rmk:tragedy}
	Each of the moves $[u]-[v]$ that
        Proposition~\ref{prop:markov-recipe} constructs from a degree
        $d$ monomial has the property that all of the columns of $u$
        and $v$, except one, are equal, and all but $d$ of the columns
        are either all 0's or all 1's.  However, not every degree        $2^{d-1}$ minimal Markov move can be constructed in this way.
        For the complex in Example~\ref{ex:2} we also need moves of
        the form
        $\left[\begin{smallmatrix}0&0&0&0&0\\0&0&1&1&1\end{smallmatrix}\right]-
        \left[\begin{smallmatrix}0&0&0&0&1\\0&0&1&1&0\end{smallmatrix}\right]$
        along with $3$ other moves obtained by replacing either of the
        first two columns with $[1\ 1]^T$ on both sides.  The other
        $4$ (which, together with the $16$ from Example~\ref{ex:2}
        give us all $24$ quadratic elements in the Markov basis) are
        obtained from
        $\left[\begin{smallmatrix}0&0&0&1&1\\ 0&0&1&0&0\end{smallmatrix}\right]-
        \left[\begin{smallmatrix}0&0&0&1&0\\ 0&0&1&0&1\end{smallmatrix}\right]$.
        To study these, one may want to look further out in the
        resolution of $I_\Delta$.

Note that the lone linear syzygy of the $2$ quadratic generators has
	multidegree $(0,0,1,1,1)$, suggesting that this is the
	source of the quadratic binomials missed by
	Proposition~\ref{prop:markov-recipe}. 
Even in the smallest degrees, 
	the minimal generators alone of $I_\Delta$ do not give enough information 	to construct all of the minimal Markov moves.
\end{rmrk}

\section{Unpredictable moves and open problems}
\label{sec:unpredictable}

The following example shows that the converse of Theorem~\ref{thm:main} does not hold. 
\begin{eg}\label{ex:alex-dual-of-5-path} 
Contrast Example~\ref{ex:4-cycle} with the complex whose Stanley-Reisner ideal is 
\[ 
	I=\ideal{x_{3} x_{4}  x_{5},x_{1} x_{2} x_{5}, x_{1} x_{4} x_{5}, x_{1} x_{2} x_{3}}, 
\]
the
Alexander dual of a 5-path.  The Betti diagram, over {any} field, is
\begin{center}
\begin{tabular}{rrrr}
 &       0 & 1 & 2\\
total: & 1 & 4 & 3\\
    0: & 1 & . & .\\
    1: & . & . & .\\
    2: & . & 4 & 3
\end{tabular}
\end{center}
Notice that only a degree $4$ generator is predicted for $I_M$. Yet
the binary model over this complex has generators in degrees $4$, $6$, $8$,
$10$ and $12$.
\end{eg}

Considering Example~\ref{ex:alex-dual-of-5-path} and
Corollary~\ref{cor:markov-linear-res}, we pose a natural question:
\begin{problem}\label{problem1}
Among those complexes with linear resolutions over every field,
which have a Markov basis concentrated in a single degree? In
particular, are there any necessary or sufficient conditions on the
Betti numbers (equivalently, on the $f$-vector of $\Delta$) to have a
Markov basis in a single degree?  
\end{problem}
In the special case where $\Delta$
is a graph, the answer to Problem~\ref{problem1} is known: for connected graphs, $M(\Delta)$ has a quadratic Markov basis
if and only if $\Delta$ is a tree, which is true if and only if
$I_\Delta$ has linear minimal free resolution, or if and only if $f(\Delta)=(1,n,n-1)$.  As
Example~\ref{ex:alex-dual-of-5-path} shows, this fails for higher dimensions.

The construction behind our main Theorem does produce all the quadratic moves for all complexes. However, we do not know if restricting {the family of complexes} will allow this procedure to produce all moves in other degrees: 
\begin{problem}
Are there any complexes where the Markov moves we construct give all the moves of that degree? 
\end{problem}
In fact, if the family of complexes is restricted, it would be interesting to see if there is a similar but more tailored construction which will produce more Markov moves then we do at the moment.

The reader will note that our proof only looks at the first terms in each row of the Betti diagram of $I_\Delta$. It is natural to expect that more can be said about the toric ideal $I_M$ if one considers the rest of the Betti diagram. In particular, we may ask the following: 
\begin{problem}
Since we have seen that the regularity of $I_\Delta$ predicts something about the generators of $I_M$, what can be said about the projective dimension? For example, does the length of each row in the Betti diagram of the Stanley-Reisner ideal predict the \emph{number} of toric Markov moves of the degree predicted by that non-zero row? 
\end{problem}

A different way to think about the proof of the main Theorem~\ref{thm:main} is that we are, essentially, linking the complex $\Delta$ repeatedly.  This, in turn, gives a series of filtrations, through which we trace the Markov moves by tracing the inclusion maps. We have not tried to obtain different kinds of Markov moves using different filtrations. To that end, we propose the following problem: 
\begin{problem}
Can additional information on the (generators of) the ideal $I_M$ be captured  by examining the homology of various filtrations of $\Delta$?
\end{problem}

The method of the proof for Theorem~\ref{thm:main} does work for all models, but provides best results for binary models only. It is not unreasonable to ask for a better bound for non-binary models. For example:
\begin{problem}
How can one obtain a better bound for models whose vertices have more then $2$ states? In particular, is there a way to change the grading on $I_\Delta$ so as to mimic the proof of Theorem~\ref{thm:main}, but provide more information about the Markov degrees?
\end{problem}

Interestingly, the Betti-row bound from Theorem~\ref{thm:main}, combined with the generalized toric fiber product construction, gives evidence in support of  Conjecture 6.6 of \cite{sethAlexThomas}, which proposes that the degrees of the Markov moves for the binary model on a triangulation of an $n$-dimensional sphere are at most $2^{n+1}$.

\medskip

We conclude this note by a computational answer to a natural question
motivated by Example~\ref{ex:2} and Remark
\ref{rmk:tragedy}:  
 how much do the predictions offered by Theorem~\ref{thm:main} 
 differ from the actual Markov bases. 

 The table below summarizes the results of computer experiments in which we used \texttt{4ti2} to compute the Markov bases and
 compared them to the resolutions obtained using \texttt{Macaulay2}.
 All resolutions were computed over the rationals.  The rows in the
 table give the number of vertices in the complex, the columns
 represent the number of degrees \emph{not} predicted. The $(i,j)$-entry in the table is the number
 of complexes on $i$ vertices for which $j$ Markov degrees are not 
 predicted by our results.

\begin{center}
\label{Table 1}
\begin{tabular}{|r||r|r|r|r|r|r|r|r|}
\hline
&\multicolumn{8}{|c|}{number of degrees not predicted}\\
\hline
$n$&0&1&2&3&4&5&6&7\\
\hline\hline
3&18 & 1 &  &  &  &  &  &  \\
\hline
4&44 & 8 &  &  &  &  &  &  \\
\hline
5&17 & 3 & 9 & 24 & 8 &  &  &  \\
\hline
6&33 & 4 & 1 & 5 &  &  & 2 & 2 \\
\hline
7&1 &  &  &  &  &  &  &  \\
\hline
8&1 &  &  &  &  &  &  & \\
\hline\hline
total & 114 &16&10&29&8&&2&2\\
\hline
\end{tabular}
\end{center}

We have also computed all complexes on $4$ vertices. Of these, only two complexes contain generators in degrees not predicted by Theorem~\ref{thm:main}.  Interestingly, the first
is the complete graph $K_4$, which, by \cite{arXiv:0810.1979v1}, must
have a generator with degree larger than $4$.  The second is homotopy
equivalent, as a topological space, to $K_4$, and thus has a similar
Betti diagram.

\section*{Acknowledgments}
This project started while the authors attended the 2007 IMA
summer program for graduate students on Applicable Algebraic Geometry
held at Texas A$\&$M University, and was continued in part during the
Algebraic Statistical Models Workshop at SAMSI in January 2009.
The authors would like to thank Alexander Engstr\"om, Thomas Kahle, Uwe Nagel, and
Seth Sullivant for helpful references, discussions and comments on a much earlier
version of this manuscript. We are also grateful to the anonymous referees for the careful reading of the paper and suggested improvements.

\def\cprime{$'$}
\providecommand{\bysame}{\leavevmode\hbox to3em{\hrulefill}\thinspace}
\providecommand{\MR}{\relax\ifhmode\unskip\space\fi MR }
\providecommand{\MRhref}[2]{%
  \href{http://www.ams.org/mathscinet-getitem?mr=#1}{#2}
}
\providecommand{\href}[2]{#2}

\end{document}